\documentclass[11pt]{amsart}     
\usepackage[foot]{amsaddr}
\usepackage{amsfonts}
\usepackage{hyperref}
\usepackage{mathrsfs}
\usepackage{amsthm}
\usepackage{amsmath}
\usepackage{amssymb}
\usepackage{graphicx}
\usepackage{tikz-cd}

\theoremstyle{plain}
\newtheorem{Th}{Theorem}
\newtheorem{Le}[Th]{Lemma}
\newtheorem{Def}[Th]{Definition}
\newtheorem{Pro}[Th]{Proposition}

  \newtheorem{thmintro}{Theorem}

\newtheorem{thm}{Theorem}
\newenvironment{thmbis}[1]
  {\renewcommand{\thethm}{\ref{#1}$'$}%
   \addtocounter{thm}{-1}%
   \begin{thm}}
  {\end{thm}}

\def \Hnc{{H}_{\mathbb{C}}^n}

\newcommand{\la}{\langle } 
\newcommand{\ra}{\rangle}
\newcommand{\omegah}{\omega_{H^n_\H}}
\newcommand{\omegac}{\omega_{H^n_\C}}
\newcommand{\classh} {c_{H^n_\H}}
\newcommand{\classc}{c_{H^n_\C}}
\renewcommand{\H}{\mathbb{H}}
\newcommand{\R}{\mathbb{R}}
\renewcommand{\P}{\mathbb{P}}
\newcommand{\C}{\mathbb{C}}

\begin{document}

\title[The Gromov norm of the quaternionic K\"ahler class]{The Gromov norm of the quaternionic K\"ahler class of $H^n_\H$  }
\author[]{Hester Pieters}
\address{Weizmann Institute of Science, Rehovot, Israel}
\email{hester.pieters@weizmann.ac.il}
\thanks{This research was supported by ISF Grant No. 535/14}
\subjclass[2010]{}
\date{}
\maketitle

\begin{abstract}
We prove that the embedding of the quaternionic disc $H^1_\H$ into quaternionic hyperbolic $n$-space $H^n_\H$ is tight and thereby obtain the value of the Gromov norm of the quaternionic K\"ahler class. 
\end{abstract}

\section{Introduction}
\label{intro}
Let $H^n_\H$ be the quaternionic hyperbolic $n$-space normalized such that the sectional curvature is bounded from below by $-1$. The quaternionic K\"ahler four-form $\omegah$ of $H_\H^n$ defines, under the van Est isomorphism, a cohomology class $[\classh]\in H^4_{c}(\mathrm{PSp}(n,1),\R)$. Recall that the \textit{Gromov norm} $\|\beta\|_\infty$ of a cohomology class $\beta$ is the semi-norm given by the infimum of the sup-norms of all cocycles representing $\beta$:
\[
\| \beta \| = \inf \{  \|b\|_\infty \mid [b]=\beta\} \in \mathbb{R}_{\geq 0} \cup \{+ \infty \}.
\]

In \cite{DT,CO} it is proven that for Hermitian symmetric spaces the Gromov norm of the K\"ahler class is equal to $\pi$ times the rank of the Hermitian symmetric space $\mathcal{X}$. This is one of the few cases where the value of the Gromov norm is known explicitely and this result has important applications related to K\"ahler rigidity and higher Teichm\"uller
theory (see \cite{BIW07, BIW09, BIW10,BIW14})
\\\\
The Gromov norm has furthermore been calculated for the volume class of hyperbolic $n$-space \cite{Gr,Th1}, the volume class of $\H^2\times \H^2$ \cite{Bu2} and for the Euler class for flat vector bundles \cite{BuMo}. Bounds for the volume form of complex hyperbolic surfaces have been obtained by the author in \cite{Pie}. The purpose of this paper is to prove

\begin{thmintro}
\label{MainTheorem}
Let $[\classh]\in H_{cb}^4(\mathrm{PSp}(n,1))$ be the quaternionic K\"ahler class. Then $\|[\classh]_b\|=v_4$, with $v_4=\frac{10\pi}{3}\arcsin\frac{1}{3}-\frac{\pi^2}{3}$ the volume of a maximal regular $4$-simplex in real hyperbolic $4$-space. 
\end{thmintro}

In fact, we will prove Theorem \ref{embedding} which we show below implies Theorem \ref{MainTheorem}. 

\begin{thmbis}{MainTheorem}
\label{embedding}
The embedding $i:H^1_\H\hookrightarrow H^n_\H$ is tight, that is 
\[
    \sup_{z^{(0)},\dots,z^{(4)}\in H^n_\H} \int_{\Delta(z^{(0)},\dots,z^{(4)})} i^*(\omegah) = \sup_{z^{(0)},\dots,z^{(4)}\in H^n_\H} \int_{\Delta(z^{(0)},\dots,z^{(4)})} \omegah,
\]
for all $z^{(0)},\dots,z^{(4)}\in H^n_\H$ and where $\Delta(z^{(0)},\dots,z^{(4)})$ is any geodesic simplex with vertices $(z^{(0)},\dots,z^{(4)})$.
\end{thmbis}
    
\begin{proof}[Theorem \ref{embedding} $\Rightarrow$ Theorem \ref{MainTheorem}]
Note that $i^*(\omegah)$ is the restriction of $\omegah$ to $H^1_\H$ which is equal to the volume form $\omega_{H^1_\H}$ on this quaternionic hyperbolic line. Therefore, Theorem \ref{embedding} implies that the sup-norm of $\classh$ is equal to the sup-norm of $\omega_{H^1_\H}$:
\begin{align*}
\|\classh\|_\infty = \|c_{H_\H^1}\|_\infty
\end{align*}
Now $H^1_\H$ is isometric to real hyperbolic $4$-space $H_\R^4$ with constant sectional curvature $-1$ (this follows from the special isomorphism of Lie algebra's $\mathrm{so}(4,1)\cong\mathrm{sp}(1,1)$, see e.g. \cite[Chapter X, \textsection 6.4]{He} ), and therefore $\omega_{H^1_\H}$ is the volume form $\omega_{H^4_\R}$ of real hyperbolic $4$-space for which 
\[
\| \omega_{H^4_\R}\|_\infty = v_4,
\]
with $v_4$ the volume of a maximal regular simplex in $H_\R^4$ \cite{HM}. Moreover, $\omega_{H_\R^4}$ realizes the Gromov norm of its cohomology class  (\cite{Gr,Bu1}) and hence
\[
\|i^*[\omega_{H^n_\H}]_b\|= \| [\omega_{H_\R^4}]_b\|=\| \omega_{H^4_\R}\|_\infty=v_4.
\]

Combining this with the fact that the Gromov norm of a cohomology class can be at most equal to the sup-norm of any of its representatives we get 
\begin{align*}
\|[\classh]_b\|_\infty&\leq \|\classh\|_\infty=\|c_{H^1_\H}\|_\infty=\|[c_{H^1_\H}]_b\|_\infty =\|[i^*(\classh)]_b\|_\infty\leq \|[\classh]_b\|_\infty,
\end{align*}
with $i^*:H_{cb}(H_\H^n,\R)\to H_{cb}(H_\H^1,\R)$ the norm-decreasing map induced by a natural inclusion $i:H^1_\H\hookrightarrow H^n_\H$. It follows that $\|[\classh]\|_\infty=\|[c_{H^1_\H}]\|_\infty=v_4$. 
\end{proof}

In the case of complex hyperbolic space there is a particular simple argument proving the corresponding Theorem \ref{embedding}, i.e. tightness of the embedding $H^1_\C \hookrightarrow H^n_\C$. We will briefly discuss this in the next section and then show in sections 3 and 4 how this argument generalizes to prove Theorem \ref{embedding}.

\subsection*{Acknowledgement}
The author would like to thank Tobias Hartnick for suggesting the problem and Tobias Hartnick and Anton Hase for useful discussions. 

\section{Complex hyperbolic space}
For the complex hyperbolic space Theorem \ref{embedding} follows from an observation of Goldman in the  proof of Theorem 7.1.11 in \cite{Gol}. Let $\Hnc$ be complex hyperbolic $n$-space with its sectional curvature bounded below by $-1$ and let $v,w,z\in\Hnc$. Denote by $\Pi_L:{H_\C^n}\to{H_\C^n}$ the orthogonal projection onto the (unique) complex line $L$ containing $v$ and $w$. Then it can be easily seen that the triples $(v,z,\Pi_L(z))$ and $(w,z,\Pi_L(z))$ are contained in totally real subspaces (it follows for example from the proof of this fact for quaternionic hyperbolic space in the next section). Since the K\"ahler form $\omegac$ vanishes on real subspaces it then follows from the cocycle relation that 
\begin{align*}
\classc(v,w,z)=\;& \classc(w,z,\Pi_L(z))-\classc(v,z,\Pi_L(z))\\
&+\classc(v,w,\Pi_L(z)) \\
=\;& \classc (v,w,\Pi_L(z))\\
<\;& \pi.
\end{align*}
The last inequality follows from the fact that on complex lines the volume form restricts to the volume form of the real hyperbolic plane with sectional curvature $-1$ and it is well known that triangles in this plane have area bounded above by $\pi$. Furthermore, the norm of the volume class is also equal to $\pi$, i.e. the norm of the volume class is realized by the volume cocycle $\classc$.

\section{Quaternionic hyperbolic space}

Let $\H^{n,1}$ be $\H^{n+1}$ with the quadratic form of signature $(n,1)$ given by 
\begin{align*}
\la \mathbf{z},\mathbf{w} \ra &=\overline{w_1}{z}_1+\dots+\overline{w_n}{z}_n-\overline{w_{n+1}}{z}_{n+1},
\end{align*}
where $\mathbf{z},\mathbf{w}\in\H^{n+1}$ are the column vectors with entries $z_1,\dots,z_{n+1}$ and $w_1,\dots,w_{n+1}$ respectively. 
Define 
\begin{align*}
V_-&= \{\mathbf{z}\in \H^{n,1} : \la \mathbf{z},\mathbf{z}\ra<0\}, \\
V_0&=\{\mathbf{z}\in \H^{n,1} : \la \mathbf{z},\mathbf{z}\ra =0\},\\
V_+&=\{\mathbf{z}\in \H^{n,1} : \la \mathbf{z},\mathbf{z}\ra>0\},
\end{align*}
and let
\[
\P: \{\mathbf{z}\in\H^{n,1} : z_{n+1}\neq 0\}\to \H^n
\] 
be the natural (right) projection given by
\[
\P:\begin{pmatrix} 	z_1 \\ 
				z_2\\
				\vdots\\
				z_n\\
				z_{n+1}
	\end{pmatrix}
\mapsto	\begin{bmatrix}	z_1z_{n+1}^{-1}\\	
					z_2z_{n+1}^{-1}\\
					\vdots\\
					z_nz_{n+1}^{-1}
		\end{bmatrix}.			
\]

The  \textit{quaternionic hyperbolic n-space} is  $H_\H^n=\P V_-$ with boundary $\partial H_\H^n =\P (V_0\setminus\{0\})$ and closure $\overline{H_\H^n}=H_\H^n\cup \partial H_\H^n$. Its isometry group is $\mathrm{PSp}(n,1)=\mathrm{Sp}(n,1)/\{\pm I\}$, where $\mathrm{Sp}(n,1)$ is the subgroup of $\mathrm{Gl}(n+1,\H)$ whose action on $\H^{n,1}$ on the left preserves the quadratic form. We normalize the metric such that the sectional curvature is pinched between $-1$ and $-1/4$. Then any embedded copy of the real hyperbolic $4$-plane has sectional curvature equal to $-1$.

\begin{Def}
An $\R$-vector subspace $V\subset \H^{n,1}$ is said to be \textit{totally real} if $\la \mathbf{w,z}\ra \in\R$ for all
$\mathbf{w},\mathbf{z}\in V$. If $V\subset \H^{n,1}$ is a totally real subspace of (real) dimension $k + 1$, and if $V\cap V_- \neq\emptyset$ then $\P(V\setminus\{0\})\cap H_\H^n$ is called a \textit{totally real subspace of dimension $k$ in $H_\H^n$}.
\end{Def}
 
For $\mathbf{v,w,z}\in\H^{n,1}$ define the \textit{Hermitian triple product} by 
\[
\la \mathbf{v,w,z\ra:=\la v,w\ra\la w,z\ra \la z,v\ra},
\]
Similar to the complex case, this Hermitian triple product can be used to define the \textit{quaternionic Cartan angular invariant} which classifies orbits of triples in the boundary of quaternionic hyperbolic space \cite[Section 3]{AK}. The only property of the Hermitian triple product we need here is the following lemma.

\begin{Le}
\label{totreal}
Let $v,w,z\in {H_\H^n}$ and let $\mathbf{v},\mathbf{w},\mathbf{z}\in \H^{n,1}$ be lifts such that $\P(\mathbf{v})=v,\P(\mathbf{w})=w$ and $\P(\mathbf{z})=z$. Then $v,w$ and $z$ are contained in a totally real subspace iff $\la \mathbf{v},\mathbf{w},\mathbf{z}\ra\in\R$. 
\end{Le}
\begin{proof}
Suppose that $v,w$ and $z$ are contained in a totally real subspace. It follows from the definition of a totally real subspace that we can choose lifts $\mathbf{v},\mathbf{w},\mathbf{z}$ such that the Hermitian products $\la \mathbf{v},\mathbf{w} \ra, \la \mathbf{w},\mathbf{z}\ra$ and $\la \mathbf{z},\mathbf{v}\ra$ are all real and therefore $\la \mathbf{v},\mathbf{w},\mathbf{z}\ra\in\R$. Then also for any other lifts $\mathbf{v}\lambda_0,\mathbf{w}\lambda_1,\mathbf{z}\lambda_2$ of $v,w$ and $z$ we have
\begin{align*}
\la\mathbf{v}\lambda_0,\mathbf{w}\lambda_1,\mathbf{z}\lambda_2\ra&=\la\mathbf{v}\lambda_0,\mathbf{w}\lambda_1\ra\la \mathbf{w}\lambda_1,\mathbf{z}\lambda_2\ra\la\mathbf{z}\lambda_2,\mathbf{v}\lambda_0\ra\\
&=\overline{\lambda_0} \la \mathbf{v},\mathbf{w}\ra \lambda_1\overline{\lambda_1}\la\mathbf{w},\mathbf{z}\ra \lambda_2\overline{\lambda_2}\la z,v\ra {\lambda_0}\\
&=\overline{\lambda_0 }\la \mathbf{v},\mathbf{w}\ra |\lambda_1|^2\la\mathbf{w},\mathbf{z}\ra |\lambda_2|^2\la z,v\ra {\lambda_0}\\
&=|\lambda_0|^2|\lambda_1|^2|\lambda_2|^2\la \mathbf{v},\mathbf{w},\mathbf{z}\ra\in\R, 
\end{align*}
where in the last equality we use that $ \la \mathbf{v},\mathbf{w}\ra |\lambda_1|^2\la\mathbf{w},\mathbf{z}\ra |\lambda_2|^2\la \mathbf{z,v}\ra \in\R$.
\\
Suppose now that $\la\mathbf{v},\mathbf{w},\mathbf{z}\ra\in\R$ and pick $\lambda_1,\lambda_2\in\H$ such that $\la \mathbf{v},\mathbf{w}\lambda_1\ra\in\R$ and $\la \mathbf{z}\lambda_2,\mathbf{v}\ra\in\R$. Then
\begin{align*}
\la \mathbf{v},\mathbf{w}\lambda_1\ra\la \mathbf{w}\lambda_1,\mathbf{z}\lambda_2\ra\la\mathbf{z}\lambda_2,\mathbf{v}\ra=|\lambda_1|^2|\lambda_2|^2\la\mathbf{v},\mathbf{w},\mathbf{z}\ra \in\R,
\end{align*}
and it follows that also $\la\mathbf{w}\lambda_1,\mathbf{z}\lambda_2\ra\in\R$ and therefore the three points are contained in a totally real subspace.
\end{proof}

Restricting to the quaternionic hyperbolic $2$-space $H^2_\H$ any quaternionic line $L$ is  given by $L=\mathbb{P}(\mathbf{c}^\perp)\cap H^2_\H$ with $\mathbf{c}\in V_+$ and $\mathbf{c}^\perp:=\{\mathbf{v}\in\H^{2,1}:\la \mathbf{v,c}\ra=0\}$. The \textit{orthogonal projection onto $L$} is the map $\Pi_L:H^2_\H\to H^2_\H$ given by 
\[
z\mapsto\P\left(\mathbf{z}-\frac{\la\mathbf{z},\mathbf{c}\ra}{\la\mathbf{c},\mathbf{c}\ra}\mathbf{c}\right),
\]
with $\mathbf{z}\in\mathbb{H}^{2,1}$ any lift of $z$, i.e. $\P(\mathbf{z})=z$.

\begin{Le}
\label{realproj}
Let $z,w\in  H^n_\H$ and let $L$ be a quaternionic line such that $z\in L$ and let $\Pi_L$ be the orthogonal projection onto $L$. Then the triple $(p,q,\Pi_L(q))$ is contained in a totally real $2$-space. 
\end{Le}
\begin{proof}
Up to the action of $\mathrm{PSp}(n,1)$ we can restrict to $H^2_\H\subset H^n_\H$ and assume that $L=\mathbb{P}(\mathbf{c}^\perp)\cap\H^2_\H$ with  
\[
\mathbf{c}=\begin{pmatrix} 0\\ 1 \\ 0\end{pmatrix}.
\]
and therefore that lifts of $z,w$ and $\Pi_L(w)$ are
\[
\mathbf{z}=\begin{pmatrix} z_1 \\ 0 \\ 1\end{pmatrix}, \mathbf{w}=\begin{pmatrix} w_1\\ w_2 \\ 1\end{pmatrix}, \;\text{and}\:\mathbf{\Pi_L(w)}=\begin{pmatrix} w_1 \\0 \\ 1\end{pmatrix},
\]
with $z_1,w_1,w_2\in \H$ such that $|z_1|^2<1$ and $|w_1|^2+|w_2|^2<1$. Then
\begin{align*}
\mathbf{\la z,w,\Pi_L(w) \ra }&= \mathbf{\la z,w \ra \la w,\Pi_L(w) \ra \la \Pi_L(w),z \ra} \\
&= (z_1\overline{w_1}-1)\cdot (|w_1|^2-1)\cdot (w_1\overline{z_1}-1) \\
&= (|w_1|^2-1)\cdot |z_1\overline{w_1}-1|^2 \in \R,
\end{align*}
and hence Lemma \ref{totreal} implies that $v,w$ and $\Pi_L(w)$ are contained in a totally real subspace. 
\end{proof}

We collect the following result from \cite{AK}:
\begin{Th}{\cite[Theorem 4.1]{AK}}
There is a quaternionic K\"ahler four-form $\omegah$ such that its restriction to any $\H$-line is its volume form. Furthermore it is a closed form and its evaluation on four vectors two of which span a totally real geodesic two plane is zero.
\end{Th}

This immediately implies
\begin{Le}
\label{realzero}
If  three points $z^{(0)},z^{(1)},z^{(2)}$ are contained in a totally real subspace of $H^n_\H$ then $\classh(z^{(0)},z^{(1)},z^{(2)},z^{(3)},z^{(4)})=0$ for all $z^{(3)},z^{(4)}\in {H^n_\H}$.
\end{Le}

\section{Proof of Theorem \ref{embedding}}

\begin{Pro}
\label{projection}
Let $(z^{(0)},z^{(1)},z^{(2)},z^{(3)},z^{(4)})\in (H^n_\H)^5$, $L_{01}$ the quaternionic line spanned by $z^{(0)}$ and $z^{(1)}$ and $\Pi_{01}$ the orthogonal projection onto $L_{01}$. Then 
\[
\classh(z^{(0)},z^{(1)},z^{(2)},z^{(3)},z^{(4)})=\classh(z^{(0)},z^{(1)},\Pi_{01}(z^{(2)}),\Pi_{01}(z^{(3)}),\Pi_{01}(z^{(4)})).
\]
\end{Pro}
\begin{proof}
We first show that projecting $z^{(4)}$ onto $L_{01}$ does not change the value of $\classh$. Indeed by the cocycle identity
\begin{align*}
0=&\delta \classh(z^{(0)},z^{(1)},z^{(2)},z^{(3)},z^{(4)},\Pi_{01}(z^{(4)}))\\
= &\classh(z^{(1)},z^{(2)},z^{(3)},z^{(4)},\Pi_{01}(z^{(4)})) - \classh(z^{(0)},z^{(2)},z^{(3)},z^{(4)},\Pi_{01}(z^{(4)}))\\
&+\classh(z^{(0)},z^{(1)},z^{(3)},z^{(4)},\Pi_{01}(z^{(4)})) -\classh(z^{(0)},z^{(1)},z^{(2)},z^{(4)},\Pi_{01}(z^{(4)})) \\
&+ \classh(z^{(0)},z^{(1)},z^{(2)},z^{(3)},\Pi_{01}(z^{(4)}))-\classh(z^{(0)},z^{(1)},z^{(2)},z^{(3)},z^{(4)}).
\end{align*}
From Lemma \ref{realzero} and Lemma \ref{realproj} it follows that the first 4 terms in the cocycle identity vanish and therefore 
\[
\classh(z^{(0)},z^{(1)},z^{(2)},z^{(3)},z^{(4)})= \classh(z^{(0)},z^{(1)},z^{(2)},z^{(3)},\Pi_{01}(z^{(4)})).
\]
Repeating the same argument to project $z^{(2)}$ and $z^{(3)}$ onto $L_{01}$ we obtain
\[
\classh(z^{(0)},z^{(1)},z^{(2)},z^{(3)},z^{(4)})=\classh(z^{(0)},z^{(1)},\Pi_{01}(z^{(2)}),\Pi_{01}(z^{(3)}),\Pi_{01}(z^{(4)})).
\]
\end{proof}

\begin{proof}[Proof of Theorem \ref{embedding}]

Proposition \ref{projection} immediately implies that for any given $5$-tuple $(z^{(0)},z^{(1)},z^{(2)},z^{(3)},z^{(4)})$ in quaternionic hyperbolic $n$-space there is a $5$-tuple $(\tilde{z}^{(0)},\tilde{z}^{(1)},\tilde{z}^{(2)},\tilde{z}^{(3)},\tilde{z}^{(4)})$ in an embedded quaternionic line $L\cong H^1_\H$ (e.g. the quaternionic line that contains $z^{(0)}$ and $z^{(1)}$) such that 
\begin{align*}
\int_{\Delta(z^{(0)},\dots,z^{(4)})} \omegah & = \int_{\Delta(\tilde{z}^{(0)},\tilde{z}^{(1)},\tilde{z}^{(2)},\tilde{z}^{(3)},\tilde{z}^{(4)})} \omegah \\
&= \int_{\Delta(\tilde{z}^{(0)},\tilde{z}^{(1)},\tilde{z}^{(2)},\tilde{z}^{(3)},\tilde{z}^{(4)})} \omegah\big|_{L} \\
& = \int_{\Delta(\tilde{z}^{(0)},\tilde{z}^{(1)},\tilde{z}^{(2)},\tilde{z}^{(3)},\tilde{z}^{(4)})} i^*(\omegah),
\end{align*}
where $i: L \hookrightarrow H^n_\H$ is an embedding of the quaternionic line containing the $5$-tuple $(\tilde{z}^{(0)},\tilde{z}^{(1)},\tilde{z}^{(2)},\tilde{z}^{(3)},\tilde{z}^{(4)})$ (and we identify $i(\tilde{z}^{(j)})$ with $\tilde{z}^{(j)}$). It follows that 
\[
    \sup_{z^{(0)},\dots,z^{(4)}\in H^n_\H} \int_{\Delta(z^{(0)},\dots,z^{(4)})} i^*(\omegah) = \sup_{z^{(0)},\dots,z^{(4)}\in H^n_\H} \int_{\Delta(z^{(0)},\dots,z^{(4)})} \omegah.
\]
\end{proof}

\end{document}